\documentclass{article}

\usepackage{amsmath}
\usepackage{amssymb}
\usepackage{amsfonts}
\usepackage{amsthm}
\usepackage{color}
\usepackage{hyperref}

\theoremstyle{plain}
\newtheorem{thm}{Theorem}[section]
\newtheorem{lem}[thm]{Lemma}
\newtheorem{prop}[thm]{Propsition}

\theoremstyle{definition}

\theoremstyle{remark}
\newtheorem{remark}[thm]{Remark}

\newtheoremstyle{exerstyle}{3pt}{3pt}{\itshape}{}{\bfseries}{}{2mm}{}
\theoremstyle{exerstyle}

\begin{document}

\title {\bf The connection between the Basel problem and a special integral}
\author{\it Haifeng Xu\thanks{\url{mailto:hfxu@yzu.edu.cn} The work is partially supported by National Natural Science Foundation of China (NSFC), Tianyuan fund for Mathematics, No. 11126046.}, Jiuru Zhou}
\date{\small\today}

\maketitle

\begin{abstract}
By using Fubini theorem or Tonelli theorem, we find that the zeta function value at 2 is equal to a special integral. Furthermore, We find that this special integral is two times of another special integral. By using this fact we obtain the relationship between Genocchi numbers and Bernoulli numbers. And get some results about Bernoulli polynomials.
\end{abstract}

%\tableofcontents


%------------------------------
% introduction

\setcounter{section}{-1}
\section{Introduction}

The Basel problem is a famous problem in mathematical analysis with relevance to number theory, first posed by Pietro Mengoli in 1644 and solved by Leonhard Euler in 1735. The problem asks for the precise summation of the reciprocals of the squares of the natural numbers, i.e. the precise sum of the infinite series:
\begin{equation}\label{eqn:zeta2}
\sum\limits_{n=1}^{+\infty}\frac{1}{n^2}=\lim_{n\rightarrow +\infty}\bigl(\frac{1}{1^2}+\frac{1}{2^2}+\cdots+\frac{1}{n^2}\bigr).
\end{equation}

Euler's original derivation of the value $\frac{\pi^2}{6}$ essentially extended observations about finite polynomials and assumed that these same properties hold true for infinite series. His idea is very intuitional but not rigorous, because his original proof should use a result which is so called Weierstrass factorization theorem introduced by Weierstrass 100 years later. For more information, please refer to \cite{Baselproblem}.

Euler generalised the problem considerably, and his ideas were taken up years later by Bernhard Riemann in his 1859 seminal paper "On the number of primes less than a given magnitude", in which he defined his zeta function and proved its basic properties.

There are more general results \cite{Lesko2003} about the progression \eqref{eqn:zeta2},
\[
\sum_{n=1}^{+\infty}\frac{1}{(n+a)^2}=\lim_{b\rightarrow a}\int_0^1\frac{t^b-t^a}{b-a}\cdot\frac{-1}{1-t}dt.
\]
Let $a=0$, it becomes the $\zeta(2)$.

Moreover, using Fourier expansion of $x^2$,
\[
x^2\sim\frac{\pi^2}{3}+4\sum_{n=1}^{+\infty}(-1)^n\frac{\cos nx}{n^2},
\]
we will get
\begin{equation}\label{eqn:half_zeta2}
\frac{\pi^2}{12}=\sum_{n=1}^{+\infty}(-1)^{n-1}\frac{1}{n^2}.
\end{equation}
In the end of section 5, we give another proof of \eqref{eqn:half_zeta2} by using the relationship of two special integrals which are introduced in section 3 and 4. Also, inspired by this, in section 6, we discuss about Bernoulli numbers and Genocchi numbers. We obtain some properties of Bernoulli numbers and Bernoulli polynomials.

%------------------------------
%section 1

\section{Basic properties}

The convergence of the infinite series \eqref{eqn:zeta2} is obvious. We can use various methods to prove it. Especially, when we consider Riemann-Zeta function, $\zeta(s)=\sum_{n=1}^{\infty}\frac{1}{n^s}$ ($s\in\mathbb{R}$), the progression diverges when $s\leq 1$, and converges when $s > 1$. Also, we can use the estimate of the partial sum of the series.
\[
\begin{split}
\sum_{n=1}^{N}\frac{1}{n^2} & < 1+\sum_{n=2}^{N}\frac{1}{n(n-1)}\\
&=1+\sum_{n=2}^{N}(\frac{1}{n-1}-\frac{1}{n})=1+1-\frac{1}{N}\rightarrow 2.
\end{split}
\]	
Or we can use the Cauchy principle. In fact, for $n>1$,
\[
\frac{1}{n^2}<\frac{1}{n(n-1)}=\frac{1}{n-1}-\frac{1}{n},
\]
thus
\[
0<\sum_{k=n+1}^{n+p}\frac{1}{k^2}<\sum_{k=n+1}^{n+p}\Bigl(\frac{1}{k-1}-\frac{1}{k}\Bigr)=\frac{1}{n}-\frac{1}{n+p}<\frac{1}{n}\rightarrow 0,\quad(n\rightarrow\infty).
\]
Then, the progression converges.

%------------------------------
%section 2

\section{Calculation of $\zeta(2)$}
There are various proofs of the Basel problem and Robin Chapman wrote a survey\cite{Chapman1999} about these. Some are elementary and some will use advanced mathematics such as Fourier analysis, complex analysis or multivariable calculus.
Here we review the method of Jiaqiang Mei \cite{mei}, which is rather elementary and easy to understand. There is also an elementary proof on the Wiki \cite{Baselproblem}.

Repeated use of the equation
\[
\frac{1}{\sin^2 x}=\frac{\cos^2 \frac{x}{2}+\sin^2 \frac{x}{2}}{4\sin^2\frac{x}{2}\cos^2\frac{x}{2}}=\frac{1}{4}\biggl[\frac{1}{\sin^2\frac{x}{2}}+\frac{1}{\sin^2\frac{\pi+x}{2}}\biggr],
\]
we get
\begin{equation}\label{eqn:2.2}
\begin{split}
\frac{1}{\sin^2 x}&=\frac{1}{4}\biggl[\frac{1}{\sin^2\frac{x}{2}}+\frac{1}{\sin^2\frac{\pi+x}{2}}\biggr]\\
&=\frac{1}{4^2}\biggl[\frac{1}{\sin^2\frac{x}{4}}+\frac{1}{\sin^2\frac{2\pi+x}{4}}+\frac{1}{\sin^2\frac{\pi+x}{4}}+\frac{1}{\sin^2\frac{3\pi+x}{4}}\biggr]\\
&=\cdots\\
&=\frac{1}{2^{2n}}\sum_{k=0}^{2^n-1}\frac{1}{\sin^2\frac{k\pi+x}{2^n}}.
\end{split}
\end{equation}
Note that
\[
\sin^2\frac{k\pi+x}{2^n}=\sin^2(\frac{k\pi+x-2^n\pi}{2^n}+\pi)=\sin^2\frac{(k-2^n)\pi+x}{2^n},
\]
we may rewrite the equation \eqref{eqn:2.2} as
\[
\begin{split}
\frac{1}{\sin^2 x}&=\frac{1}{2^{2n}}\sum_{k=-2^{n-1}}^{2^{n-1}-1}\frac{1}{\sin^2\frac{x+k\pi}{2^n}}\\
&=E_n+\sum_{k=-2^{n-1}}^{2^{n-1}-1}\frac{1}{(x+k\pi)^2},
\end{split}
\]
where
\[
E_n=\frac{1}{2^{2n}}\sum_{k=-2^{n-1}}^{2^{n-1}-1}\biggl[\frac{1}{\sin^2\frac{x+k\pi}{2^n}}-\frac{1}{(\frac{x+k\pi}{2^n})^2}\biggr].
\]
Using the inequality
\[
0 < \frac{1}{\sin^2 x}-\frac{1}{x^2}=1+\frac{\cos^2 x}{\sin^2 x}-\frac{1}{x^2} < 1,\quad\forall\ x\in [-\frac{\pi}{2},\frac{\pi}{2}],
\]
we get the estimation
\[
0 < E_n < \frac{1}{2^{2n}}\cdot 2^n=\frac{1}{2^n},\quad\forall\ x\in [0,\frac{\pi}{2}].
\]
Let $n\rightarrow\infty$, we obtain the following equation
\[
\frac{1}{\sin^2 x}=\sum_{k\in\mathbb{Z}}\frac{1}{(x+k\pi)^2},\quad x\neq n\pi.
\]
The above progression is uniformly convergent in any closed interval not containing  $\{n\pi\}$ and can be written as
\[
\frac{1}{\sin^2 x}=\frac{1}{x^2}+\sum_{n=1}^{\infty}\biggl[\frac{1}{(x+n\pi)^2}+\frac{1}{(x-n\pi)^2}\biggr],\quad x\neq k\pi.
\]
Especially, we have
\[
\frac{1}{3}=\lim_{x\rightarrow 0}(\frac{1}{\sin^2 x}-\frac{1}{x^2})=2\sum_{n=1}^{\infty}\frac{1}{(n\pi)^2},
\]
Therefore,
\[
\zeta(2)=\sum_{n=1}^{\infty}\frac{1}{n^2}=\frac{\pi^2}{6}.
\]

%------------------------------
%section 3

\section{As a special case of power series}
For the power series $\sum_{n=1}^{\infty}\frac{2^n}{n^2}x^n$, we calculate the domain of convergence. Since
\[
\lim_{n\rightarrow\infty}\sqrt[n]{|a_n|}=\lim_{n\rightarrow\infty}\sqrt[n]{\frac{2^n}{n^2}}=2,
\]
the radius of convergence equals $R=\frac{1}{2}$. If $x=\frac{1}{2}$, the power series becomes the progression $\sum_{n=1}^{\infty}\frac{1}{n^2}$ which is convergent. If $x=-\frac{1}{2}$, then the power series becomes $\sum_{n=1}^{\infty}(-1)^n\frac{1}{n^2}$ which is also convergent. Therefore, the domain of convergence is $[-\frac{1}{2},\frac{1}{2}]$.

Suppose $S(x)=\sum_{n=1}^{\infty}\frac{2^n}{n^2}x^n$. We can do the derivation item by item in the inteval $(-\frac{1}{2},\frac{1}{2})$. That is,
\[
S'(x)=\sum_{n=1}^{\infty}\frac{2^n}{n}x^{n-1}.
\]
Multiply both sides by $x$,
\[
xS'(x)=\sum_{n=1}^{\infty}\frac{2^n}{n}x^n.
\]
Derivate both sides again, we get
\[
S'(x)+xS''(x)=\sum_{n=1}^{\infty}2^n x^{n-1}.
\]
Thus,
\[
\frac{1}{2}(S'(x)+xS''(x))=\sum_{n=1}^{\infty}(2x)^{n-1}=\frac{1}{1-2x}.
\]
We obtain a second order ordinary differential equation
\begin{equation}\label{eqn:3.3}
S'(x)+xS''(x)=\frac{2}{1-2x}.
\end{equation}
If we set $y(x)=S'(x)$, the equation \eqref{eqn:3.3} is converted to a first order equation,
\[
y+xy'=\frac{2}{1-2x}.
\]
Multiplie both side by $dx$,
\[
ydx+xdy=\frac{2}{1-2x}dx.
\]
Let $h(x)=xy(x)$, we have
\[
dh(x)=\frac{2}{1-2x}dx.
\]
Then
\[
h(x)=-\ln(1-2x)+C.
\]
Using the initial conditions $y(0)=S'(0)=2$, $h(0)=0y(0)=0$, we have
\[
h(x)=-\ln(1-2x).
\]
Then, if $x\neq 0$,
\[
y(x)=-\frac{\ln(1-2x)}{x}.
\]
That is
\[
S'(x)=-\frac{\ln(1-2x)}{x}.
\]
Note that $S(0)=0$. Then,
\[
S(x)=S(x)-S(0)=\int_0^x S'(t)dt=-\int_0^x \frac{\ln(1-2t)}{t}dt.
\]
Particularly,
\[
\begin{split}
S(\frac{1}{2})&=-\int_0^{1/2} \frac{\ln(1-2t)}{t}dt\\
&=-\int_{0}^{1}\frac{\ln(1-u)}{u}du\\
&=-\int_{0}^{1}\frac{\ln u}{1-u}du.
\end{split}
\]
Therefore, for the improper integral $\int_{0}^{1}\frac{\ln u}{1-u}du$, we know its value is equal to $-\frac{\pi^2}{6}$.

%------------------------------
%section 4

\section{From the special integral to the Basel problem}

In this section, we will calculate the special integral arised in the last section, i.e.
\[
\int_{0}^{1}\frac{\ln u}{1-u}du.
\]
For $n\geqslant 0$,
\[
\begin{split}
\int_{0}^{1}t^n\ln t dt&=\frac{1}{n+1}\int_{0}^{1}\ln t dt^{n+1}\\
&=\frac{1}{n+1}\biggl[t^{n+1}\ln t\biggl|_{0}^{1}-\int_{0}^{1}t^{n+1}d\ln t\biggr]\\
&=\frac{1}{n+1}\biggl[0-\lim_{t\rightarrow 0^{+}}t^{n+1}\ln t-\int_{0}^{1}t^n dt\biggr]\\
&=-\frac{1}{(n+1)^2}.
\end{split}
\]
Thus
\[
\sum_{n=1}^{\infty}\frac{1}{n^2}=\sum_{n=1}^{\infty}\int_{0}^{1}t^{n-1}|\ln t|dt.
\]
Let $f_n(t)=t^{n-1}|\ln t|$, $t\in(0,1]$, $n=1,2,\ldots$. Obviously, $f_n\in L^1([0,1],|\cdot|)$, where $|\cdot|$ is Lebesgue measure. For simplicity, we denote $L^1([0,1],|\cdot|)$ by $L^1$.

By Minkowski inequality \cite{Grafakos}, we have
\begin{equation}\label{eqn:4.1}
\biggl\|\sum_{n=1}^{\infty}f_n\biggr\|_{L^1}\leqslant\sum_{n=1}^{\infty}\|f_n\|_{L^1}.
\end{equation}
Then,
\[
\int_{0}^{1}\Bigl(\sum_{n=1}^{\infty}t^{n-1}|\ln t|\Bigr)dt\leqslant\sum_{n=1}^{\infty}\int_{0}^{1}t^{n-1}|\ln t|dt=\sum_{n=1}^{\infty}\frac{1}{n^2}=\zeta(2).
\]
We will prove the equality holds in our case. First we have the following lemma.

\begin{lem}\label{lem:4.1}
$\|f+g\|_{L^1}=\|f\|_{L^1}+\|g\|_{L^1}$ iff there is a real valued function $h$ that is nonnegative a.e. such that when both $f$ and $g$ are not 0 then $g = hf$ a.e.
\end{lem}
\begin{proof}
Please refer to \cite{Burkard}.
\end{proof}

Lemma \ref{lem:4.1} can be generalized to infinite summation case.
\begin{lem}
$\|\sum_{n=1}^{\infty}f_n\|_{L^1}=\sum_{n=1}^{\infty}\|f_n\|_{L^1}$ iff there is a real valued function $f$ and a series of real valued functions $g_n$ which have the same signs such that $f_n = g_n f$ a.e.
\end{lem}
\begin{proof}
First, by induction, the lemma holds for finite sum. That is
\[
\|\sum_{n=1}^{N}f_n\|_{L^1}=\sum_{n=1}^{N}\|f_n\|_{L^1}.
\]
Let $E:=\{x\mid f_n(x) \text{have the same sign}\}$. Then,
\[
\biggl\|\sum_{n=1}^{\infty}f_n\biggr\|_{L^1(E)}\geqslant\biggl\|\sum_{n=1}^{N}f_n\biggr\|_{L^1(E)}=\sum_{n=1}^{N}\|f_n\|_{L^1(E)}\rightarrow\sum_{n=1}^{\infty}\|f_n\|_{L^1(E)}.
\]
Since the measure of $E^c$ is zero, we have
\[
\biggl\|\sum_{n=1}^{\infty}f_n\biggr\|_{L^1}\geqslant\sum_{n=1}^{\infty}\|f_n\|_{L^1}.
\]
Combine equation \eqref{eqn:4.1}, we complete the proof.
\end{proof}

On the other hand, we observe that, for $t\in (0,1)$
\[
\frac{\ln t}{1-t}=\ln t\Bigl(\sum_{n=0}^{\infty}t^n\Bigr)=\sum_{n=1}^{\infty}t^{n-1}\ln t.
\]
Then we get
\[
\int_{0}^{1}\frac{|\ln t|}{1-t}dt=\int_{0}^{1}\Bigl(\sum_{n=1}^{\infty}t^{n-1}|\ln t|\Bigr)dt=\sum_{n=1}^{\infty}\int_{0}^{1}t^{n-1}|\ln t|dt=\sum_{n=1}^{\infty}\frac{1}{n^2}=\zeta(2).
\]

We give a remark about the second equality of the above equation. It can be infered by Fubini theorem or Tonelli's theorem \cite{FubiniTheorem}.

Infact,
\[
\int_{0}^{1}\sum_{n=1}^{\infty}t^{n-1}|\ln t|dt=\int_{[0,1]}\int_{\mathbb{N}}t^{n-1}|\ln t|d\mu(n)d\nu(t),
\]
where $\mu$ is the counting measure on $\mathbb{N}$, and $\nu$ is the Lebesgue measure on $\mathbb{R}$. Obviously, they are both $\sigma$-finite measures. And since $t^{n-1}|\ln t|$ is non-negative, by Tonelli's theorem,
\[
\int_{[0,1]}\int_{\mathbb{N}}t^{n-1}|\ln t|d\mu(n)d\nu(t)=\int_{\mathbb{N}}\int_{[0,1]}t^{n-1}|\ln t|d\mu(n)d\nu(t)=\sum_{n=1}^{\infty}\int_{0}^{1}t^{n-1}|\ln t|dt.
\]

There are other ways to get this relationship from this special integral to $\zeta(2)$. First, recall the lemma established by James P. Lesko and Wendy D. Smith \cite{Lesko2003}.
\begin{lem}
For $r\in [-1,1)$, $a>0$ and $b\geqslant 0$, we have
\begin{equation}\label{eqn:5}
\sum_{n=1}^{\infty}\frac{r^n}{an+b}=\frac{1}{a}\int_{0}^{1}\frac{ru^{\frac{b}{a}}}{1-ru}du.
\end{equation}
Especially, when $a=1$ and $b=0$, \eqref{eqn:5} yields
\[
\sum_{n=1}^{\infty}\frac{r^n}{n}=\int_{0}^{1}\frac{r}{1-ru}du=-\ln(1-r).
\]
\end{lem}

By this lemma,
\[
\begin{split}
\int_{0}^{1}\frac{|\ln t|}{1-t}dt&=(-1)\int_{0}^{1}\frac{\ln(1-r)}{r}dr\\
&=(-1)\int_{0}^{1}\biggl[\frac{1}{r}\cdot(-1)\int_{0}^{1}\frac{r}{1-ru}du\biggr]dr\\
&=\int_{0}^{1}\int_{0}^{1}\frac{1}{1-ru}dudr.
\end{split}
\]
Then by monotone convergence theorem, it equals to $\sum_{n=1}^{\infty}\frac{1}{n^2}$. See \cite{Chapman1999}.

Or, we can do it in this way,
\[
\begin{split}
\int_{0}^{1}\frac{|\ln t|}{1-t}dt&=(-1)\int_{0}^{1}\frac{\ln(1-r)}{r}dr\\
&=(-1)\int_{0}^{1}\biggl[\frac{1}{r}\cdot(-1)\sum_{n=1}^{\infty}\frac{r^n}{n}\biggr]dr\\
&=\int_{0}^{1}\sum_{n=1}^{\infty}\frac{r^{n-1}}{n}dr\\
&=\sum_{n=1}^{\infty}\frac{1}{n}\int_{0}^{1}r^{n-1}dr\\
&=\sum_{n=1}^{\infty}\frac{1}{n^2}.
\end{split}
\]

%------------------------------
%section 5

\section{Relationship between two special integrals}

We will use a result from \cite[Exer 20]{Polya1}
\begin{lem}\label{lem:5.1}
Assume that the function $f(x)$ is monotone on the interval $(0,1)$. It need not be bounded at the points $x=0$, $x=1$; we assume however that the improper integral $\int_{0}^{1}f(x)dx$ exists. Under these conditions,
\[
\lim_{n\rightarrow+\infty}\frac{1}{n}\Bigl[f(\frac{1}{n})+f(\frac{2}{n})+\cdots f(\frac{n-1}{n})\Bigr]=\int_{0}^{1}f(x)dx.
\]
\end{lem}

Then for our case, the integral $\int_{0}^{1}\frac{\ln t}{1-t}dt$ exists and satisfies the conditions. In fact, let $f(t)=\frac{\ln t}{1-t}$, $t\in(0,1)$, then
\[
\begin{split}
f'(t)&=\frac{\frac{1}{t}(1-t)-\ln t\cdot(-1)}{(1-t)^2}\\
&=\frac{1-t+t\ln t}{t(1-t)^2}.
\end{split}
\]
Let $g(t)=1-t+t\ln t$, then
\[
g'(t)=-1+\ln t+1=\ln t < 0.
\]
Since $\lim\limits_{t\rightarrow 1}g(t)=0$, we have $g(t) > 0$ for $t\in(0,1)$. So $f(t)$ is monotone on the interval $(0,1)$. Then by Lemma \ref{lem:5.1}, we have
\[
\begin{split}
\int_{0}^{1}\frac{\ln t}{1-t}dt&=\lim_{n\rightarrow\infty}\frac{1}{n}\sum_{k=1}^{n-1}\frac{\ln\frac{k}{n}}{1-\frac{k}{n}}\\
&=\lim_{n\rightarrow\infty}\sum_{k=1}^{n-1}\frac{\ln k-\ln n}{n-k}.
\end{split}
\]

If we consider the improper integral $\int_{0}^{1}\frac{\ln (1-t)}{t}dt$. Let $f(t)=\frac{\ln(1-t)}{t}$, $t\in(0,1)$. Then
\[
f'(t)=\frac{\frac{-t}{1-t}-\ln(1-t)}{t^2}.
\]
Let $g(t)=\frac{-t}{1-t}-\ln(1-t)$, then $g'(t)=\frac{-t}{(1-t)^2}<0$. Since
\[
\lim_{t\rightarrow 0^+}g(t)=\lim_{t\rightarrow 0^+}\Bigl(\frac{-t}{1-t}-\ln(1-t)\Bigr)=0,
\]
$g(t)<0$. Thus, $f'(t)<0$. So $f(t)$ is also monotone on the interval $(0,1)$. By Lemma \ref{lem:5.1}, we have
\[
\begin{split}
\int_{0}^{1}\frac{\ln(1-t)}{t}dt&=\lim_{n\rightarrow\infty}\frac{1}{n}\sum_{k=1}^{n-1}\frac{\ln(1-\frac{k}{n})}{\frac{k}{n}}\\
&=\lim_{n\rightarrow\infty}\sum_{k=1}^{n-1}\frac{\ln(1-\frac{k}{n})}{k}\\
&=\lim_{n\rightarrow\infty}\ln\prod_{k=1}^{n-1}(1-\frac{k}{n})^{\frac{1}{k}}\\
&=\ln\lim_{n\rightarrow\infty}\prod_{k=1}^{n-1}(1-\frac{k}{n})^{\frac{1}{k}}.
\end{split}
\]

Next, we deduce the following equation and give another discription of \eqref{eqn:half_zeta2},

\begin{equation}\label{eqn:twoIntegral}
\int_{0}^{1}\frac{\ln t}{1-t}dt=2\int_{0}^{1}\frac{\ln t}{1+t}dt.
\end{equation}

\begin{lem}\label{lem:5.2}
Let $h(x)=\int_{0}^{x}\frac{\ln(1-t)}{t}dt$, $x\in[-1,1]$, then $h(x)+h(-x)=\frac{1}{2}h(x^2)$.
\end{lem}
\begin{proof}
Let $\varphi(x)=h(x)+h(-x)-\frac{1}{2}h(x^2)$, it is easy to note that its derivate $\varphi'(x)$ equals to zero and $\varphi(0)=0$.
\end{proof}

By changing variables,
\[
h(x)=\int_{0}^{x}\frac{\ln(1-t)}{t}dt\stackrel{u=-t}{=}\int_{0}^{-x}\frac{\ln(1+u)}{u}du.
\]
Then
\[
\begin{split}
h(-1)&=\int_{0}^{1}\frac{\ln(1+u)}{u}du\\
&=\int_{0}^{1}\ln(1+u)d\ln u\\
&=\ln(1+u)\ln u\biggr|_{0}^{1}-\int_{0}^{1}\ln ud\ln(1+u)\\
&=-\int_{0}^{1}\frac{\ln u}{1+u}du.
\end{split}
\]
Here, note that $\lim\limits_{u\rightarrow 0^+}\ln(1+u)\ln u=0$. By Lemma \ref{lem:5.2},
\[
\int_{0}^{1}\frac{\ln t}{1-t}dt=\int_{0}^{1}\frac{\ln(1-t)}{t}dt=h(1)=-2h(-1)=2\int_{0}^{1}\frac{\ln t}{1+t}dt.
\]

For the integral $\int_{0}^{1}\frac{\ln t}{1+t}dt$, we have a relevant result.
\begin{lem}
Let $h(x)=\int_{1}^{x}\frac{\ln t}{1+t}dt$, $x>0$, then $h(x)+h(\frac{1}{x})=\frac{1}{2}(\ln x)^2$.
\end{lem}
\begin{proof}
\[
h(x)=\int_{1}^{x}\frac{\ln t}{1+t}dt=\int_{1}^{x}\ln t d\ln(1+t)=\ln x\ln(1+x)-\int_{1}^{x}\frac{\ln(1+t)}{t}dt,
\]
then,
\[
h(\frac{1}{x})=\ln\frac{1}{x}\ln(1+\frac{1}{x})-\int_{1}^{\frac{1}{x}}\frac{\ln(1+t)}{t}dt.
\]
Hence,
\[
\begin{split}
h(x)+h(\frac{1}{x})&=\ln x\ln(1+x)-\ln x\bigl(\ln(1+x)-\ln x\bigr)-\biggl[\int_{1}^{x}\frac{\ln(1+t)}{t}dt+\int_{1}^{\frac{1}{x}}\frac{\ln(1+t)}{t}dt\biggr]\\
&=(\ln x)^2-\frac{1}{2}(\ln x)^2=\frac{1}{2}(\ln x)^2.
\end{split}
\]
The second equality holds, because
\[
\begin{split}
\int_{1}^{x}\frac{\ln(1+t)}{t}dt&\stackrel{u=1/t}{=}-\int_{1}^{\frac{1}{x}}\frac{\ln(1+t)}{t}dt+\frac{1}{2}(\ln x)^2.
\end{split}
\]
\end{proof}

Observe that
\[
\begin{split}
\int_{0}^{1}\frac{\ln(1+t)}{t}dt&=\int_{0}^{1}\frac{1}{t}\sum_{n=1}^{\infty}(-1)^{n-1}\frac{t^n}{n}dt\\
&=\sum_{n=1}^{\infty}\frac{(-1)^{n-1}}{n}\int_{0}^{1}t^{n-1}dt\\
&=\sum_{n=1}^{\infty}(-1)^{n-1}\frac{1}{n^2}.
\end{split}
\]
Thus
\[
\sum_{n=1}^{\infty}(-1)^{n-1}\frac{1}{n^2}=\int_{0}^{1}\frac{\ln(1+t)}{t}dt=-\int_{0}^{1}\frac{\ln t}{1+t}dt=-\frac{1}{2}\int_{0}^{1}\frac{\ln t}{1-t}dt=\frac{1}{2}\zeta(2).
\]

Applying the same argument in beginning of this section, we get
\[
\int_{0}^{1}\frac{\ln t}{1+t}dt=\lim_{n\rightarrow\infty}\sum_{k=1}^{n-1}\frac{\ln k-\ln n}{n+k},
\]
and
\[
\int_{0}^{1}\frac{\ln(1+t)}{t}dt=\ln\lim_{n\rightarrow\infty}\prod_{k=1}^{n-1}(1+\frac{k}{n})^{\frac{1}{k}}.
\]

\begin{remark}
It must be very interesting if we could calculate the integral $\int_{0}^{1}\frac{\ln t}{1-t}dt$ not using the progression $\sum_{n=1}^{\infty}\frac{1}{n^2}$.
\end{remark}

%------------------------------
%section 6

\section{Bernoulli numbers and Bernoulli polynomials}

Recall some facts of Bernoulli numbers, and for more information, please refer to \cite{GTM112,mei,Polya2,Stein1,Stein2}.

The Bernoulli numbers $B_n$ are defined by the power series expansion
\[
\frac{z}{e^z-1}=\sum_{n=0}^{\infty}\frac{B_n}{n!}z^n.
\]
Then
\[
z=(e^z-1)(\sum_{n=0}^{\infty}\frac{B_n}{n!}z^n)=(\sum_{m=1}^{\infty}\frac{z^m}{m!})(\sum_{n=0}^{\infty}\frac{B_n}{n!}z^n)=\sum_{n=1}^{\infty}\frac{1}{n!}\biggl[\sum_{k=0}^{n-1}\binom{n}{k}B_k\biggr]z^n.
\]
Thus we get a recursion formula for the Bernoulli numbers, namely
\[
\frac{B_0}{n!0!}+\frac{B_1}{(n-1)!1!}+\cdots+\frac{B_{n-1}}{1!(n-1)!}=
\begin{cases}
1, & \text{if}\ n=1,\\
0, & \text{if}\ n>1.
\end{cases}
\]
We get $B_0=1$. From the identity
\[
\frac{z}{e^z-1}+\frac{z}{2}=\frac{z}{2}\cdot\frac{e^z+1}{e^z-1}=\frac{z}{2}\cdot\frac{e^{z/2}+e^{-z/2}}{e^{z/2}-e^{-z/2}}=\frac{\frac{z}{2}}{\tanh\frac{z}{2}},
\]
the function $\frac{z}{e^z-1}+\frac{z}{2}$ is an even function. Hence it has only even terms in its power series expansion.
\[
\frac{z}{e^z-1}=1-\frac{z}{2}+\sum_{n=1}^{\infty}\frac{B_{2n}}{(2n)!}z^{2n}.
\]
Sometimes, people prefer to use $\mathcal{B}_n$ to denote $(-1)^{n-1}B_{2n}$, then
\[
\frac{z}{e^z-1}=1-\frac{z}{2}+\sum_{n=1}^{\infty}\frac{(-1)^{n-1}\mathcal{B}_n}{(2n)!}z^{2n}.
\]
We have various ways to get the important equation
\begin{lem}
\[\label{eqn:ZetaBernoulli}
\zeta(2n)=(-1)^{n-1}\frac{2^{2n-1}\pi^{2n}}{(2n)!}B_{2n}=\frac{2^{2n-1}\pi^{2n}}{(2n)!}\mathcal{B}_n.
\]
\end{lem}
\begin{proof}
By replacing $z$ by $2\pi iz$ in the identity
\[
\frac{z}{2}\cdot\frac{e^{z/2}+e^{-z/2}}{e^{z/2}-e^{-z/2}}=\sum_{n=0}^{\infty}\frac{B_{2n}}{(2n)!}z^{2n},
\]
we get
\[
\pi z\cot\pi z=\sum_{n=0}^{\infty}(-1)^n\frac{(2\pi)^{2n}}{(2n)!}B_{2n}z^{2n}.
\]
Then by taking the logarithmic derivative of the product expansion for the sine,
\[
\sin\pi z=\pi z\prod_{n=1}^{\infty}(1-\frac{z}{n})(1+\frac{z}{n}),
\]
we get the expansion of $\pi z\cot\pi z$.
\[
\pi z\cot\pi z=1+\sum_{n=1}^{\infty}\biggl[\frac{z}{z-n}+\frac{z}{z+n}\biggr]=1-2\sum_{n=1}^{\infty}\zeta(2n)z^{2n}.
\]
Comparing the coefficents of $z^{2n}$, we get \eqref{eqn:ZetaBernoulli}.

Another way is to take the logarithmic derivative of the identity \cite{mei}
\[
\frac{\sinh x}{x}=\prod_{k=1}^{\infty}(\frac{x^2}{k^2\pi^2}+1),\quad x\in\mathbb{R},
\]
which yields
\[
x\coth x =1+\sum_{m=0}^{\infty}\frac{2(-1)^{m}\zeta(2m+2)}{\pi^{2m+2}}x^{2m+2}.
\]
Then we get \eqref{eqn:ZetaBernoulli} by comparing with
\[
x\coth x=\frac{x}{\tanh x}=\frac{2x}{e^{2x}-1}+x=\sum_{n=0}^{\infty}\frac{2^{2n}B_{2n}}{(2n)!}x^{2n}.
\]
\end{proof}

\begin{prop}
\[
\sum_{n=1}^{\infty}(-1)^{n-1}\mathcal{B}_n=\frac{\pi^2}{6}-\frac{3}{2}.
\]
\end{prop}
\begin{proof}
\[
\begin{split}
\int_{0}^{1}\frac{|\ln t|}{1-t}dt&\stackrel{\tiny s=\ln t}{=}\int_{-\infty}^{0}\frac{se^s}{e^s-1}ds\\
&=\int_{-\infty}^{0}e^s\biggl[1-\frac{s}{2}+\sum_{n=1}^{\infty}\frac{(-1)^{n-1}\mathcal{B}_n}{(2n)!}s^{2n}\biggr]ds\\
&=\int_{-\infty}^{0}e^sds-\frac{1}{2}\int_{-\infty}^{0}sde^s+\int_{-\infty}^{0}e^s\sum_{n=1}^{\infty}\frac{(-1)^{n-1}\mathcal{B}_n}{(2n)!}s^{2n}ds\\
&=1-\frac{1}{2}\Bigl[se^s\Bigl|_{-\infty}^{0}-\int_{-\infty}^{0}e^sds\Bigr]+\sum_{n=1}^{\infty}\frac{(-1)^{n-1}\mathcal{B}_n}{(2n)!}\int_{-\infty}^{0}s^{2n}e^sds\\
&=\frac{3}{2}+\sum_{n=1}^{\infty}\frac{(-1)^{n-1}\mathcal{B}_n}{(2n)!}\cdot I_{2n},
\end{split}
\]
where
\[
I_{2n}=\int_{-\infty}^{0}s^{2n}e^sds.
\]
It is easy to prove that $I_{k}=(-1)^k\cdot k!$ by induction. Therefore,
\[
\int_{0}^{1}\frac{|\ln t|}{1-t}dt=\frac{3}{2}+\sum_{n=1}^{\infty}(-1)^{n-1}\mathcal{B}_n.
\]
\end{proof}

Let us consider the expansion
\[
\frac{z}{e^z+1}=\sum_{n=0}^{\infty}\frac{G_n}{n!}z^n,
\]
where $G_n$ are Genocchi numbers. Then
\[
z=(e^z+1)(\sum_{n=0}^{\infty}\frac{G_n}{n!}z^n)=(2+\sum_{m=1}^{\infty}\frac{z^m}{m!})(\sum_{n=0}^{\infty}\frac{G_n}{n!}z^n)=\sum_{n=0}^{\infty}\frac{2G_n}{n!}z^n+\sum_{n=1}^{\infty}\frac{1}{n!}\biggl[\sum_{k=0}^{n-1}\binom{n}{k}G_k\biggr]z^n.
\]
Thus, $G_0=0$ and $2G_1+G_0=1$, which infers that $G_1=-\frac{1}{2}$. For $n>1$,
\[
\frac{2G_n}{n!}+\sum_{k=0}^{n-1}\frac{1}{n!}\binom{n}{k}G_k=0.
\]
That is,
\begin{equation}\label{eqn:G_n}
G_n=-\frac{1}{2}\sum_{k=0}^{n-1}\binom{n}{k}G_k,\quad (n>1).
\end{equation}

Note that
\[
\frac{z}{e^z-1}-\frac{z}{e^z+1}=\frac{2z}{e^{2z}-1}.
\]
Taking expansion,
\[
\sum_{n=0}^{\infty}\frac{B_n}{n!}z^n-\sum_{n=0}^{\infty}\frac{G_n}{n!}z^n=\sum_{n=0}^{\infty}\frac{B_n}{n!}(2z)^n,
\]
we have
\begin{equation}\label{eqn:G_n_and_B_n}
G_n=-(2^n-1)B_n.
\end{equation}
This give a quick way to compute Bernoulli numbers since in \eqref{eqn:G_n} we have $\binom{n+1}{k}=\frac{n+1}{n+1-k}\binom{n}{k}$.

If let $\mathcal{G}_n$ denote $(-1)^{n-1}G_{2n}$, then
\[
\mathcal{G}_n=-(2^{2n}-1)\mathcal{B}_n.
\]

\begin{prop}
\[
\sum_{n=1}^{\infty}(-1)^{n}(2^n-1)B_n=\sum_{n=1}^{\infty}(-1)^{n-1}G_n=\frac{\pi^2}{12}.
\]
\end{prop}
\begin{proof}
By changing variables,
\[
\int_{0}^{1}\frac{\ln t}{1+t}dt=\int_{-\infty}^{0}\frac{xe^x}{1+e^x}dx.
\]
Note that
\[
\frac{x}{e^x+1}=\sum_{n=1}^{+\infty}\frac{G_n}{n!}x^n,
\]
then
\[
\begin{split}
\int_{0}^{1}\frac{\ln t}{1+t}dt&=\int_{-\infty}^{0}\Bigl(\sum_{n=1}^{+\infty}\frac{G_n}{n!}x^n\Bigr)e^xdx\\
&=\sum_{n=1}^{+\infty}\frac{G_n}{n!}\int_{-\infty}^{0}x^ne^x dx.
\end{split}
\]
Set $I_n=\int_{-\infty}^{0}x^ne^x dx$, then $I_n=(-n)I_{n-1}$. By induction, we prove that $I_n=(-1)^n n!$. Therefore,
\[
\int_{0}^{1}\frac{\ln t}{1+t}dt=\sum_{n=1}^{\infty}(-1)^{n}G_n.
\]
Together with \eqref{eqn:twoIntegral}, we have
\[
\sum_{n=1}^{\infty}(-1)^{n-1}G_n=\frac{\pi^2}{12}.
\]
\end{proof}
\begin{remark}
Since $B_1=1$ and $G_1=-\frac{1}{2}$, the proposition can also be written as
\[
\sum_{n=1}^{\infty}(2^{2n}-1)\mathcal{B}_n=\frac{\pi^2}{12}+1,\quad\sum_{n=1}^{\infty}\mathcal{G}_n=-\frac{\pi^2}{12}-\frac{1}{2}.
\]
\end{remark}

Bernoulli polynomials $B_n(x)$ are defined by the formula
\[
\frac{ze^{xz}}{e^z-1}=\sum_{n=0}^{\infty}\frac{B_n(x)}{n!}z^n.
\]
The functions $B_n(x)$ are polynomials in $x$ and
\[
B_n(x)=\sum_{k=0}^{n}\binom{n}{k}B_{n-k}x^{k}.
\]
Similarly, we define $G_n(x)$ by the formula
\begin{equation}\label{eqn:Def_of_Dnx}
\frac{ze^{xz}}{e^z+1}=\sum_{n=0}^{\infty}\frac{G_n(x)}{n!}z^n.
\end{equation}
The functions $G_n(x)$ are polynomials in $x$. In fact,
\[
\sum_{n=0}^{\infty}\frac{G_n(x)}{n!}z^n=\Bigl(\frac{z}{e^z+1}\Bigr)e^{xz}=\biggl(\sum_{n=0}^{\infty}\frac{G_n}{n!}z^n\biggr)\biggl(\sum_{n=0}^{\infty}\frac{x^n}{n!}z^n\biggr).
\]
Comparing the coefficients of $z^n$, we have
\begin{equation}\label{eqn:Dnx}
G_n(x)=\sum_{k=0}^{n}\binom{n}{k}G_{n-k}x^{k}=\sum_{k=0}^{n}\binom{n}{k}G_{k}x^{n-k},
\end{equation}
and $G_n(0)=G_n$. Using \eqref{eqn:G_n}, for $n>1$, we have
\[
G_n(1)=\sum_{k=0}^{n}\binom{n}{k}G_{k}=\sum_{k=0}^{n-1}\binom{n}{k}G_{k}+G_n=-2G_n+G_n=-G_n.
\]
On the other hand, by definition,
\[
\frac{ze^{xz}}{e^z+1}=\sum_{n=0}^{\infty}\frac{G_n(x)}{n!}z^n,\quad\frac{ze^{(x+1)z}}{e^z+1}=\sum_{n=0}^{\infty}\frac{G_n(x+1)}{n!}z^n.
\]
Do the addition,
\[
\sum_{n=0}^{\infty}\frac{G_n(x+1)+G_n(x)}{n!}z^n=ze^{xz}=\sum_{n=0}^{\infty}\frac{x^n}{n!}z^{n+1}.
\]
Comparing the coefficients of $z^n$, we have
\[
G_{k}(x+1)+G_{k}(x)=kx^{k-1},\quad k\geqslant 2.
\]
Let $x=1,2,3,\ldots,n$ and summation these equations, we get
\[
G_{k}(1)+2\bigl(G_{k}(2)+\cdots+G_{k}(n)\bigr)+G_{k}(n+1)=k\sum_{i=1}^{n}i^{k-1}.
\]
From the equation $G_{k}(1)+G_{k}(0)=0$, we infer that $G_{k}(1)=-G_{k}(0)=-G_{k}$.
\[
\begin{aligned}
G_{k}(2)&=-G_{k}(1)+k=k+G_{k},\\
G_{k}(3)&=-G_{k}(2)+k2^{k-1}=k2^{k-1}-k-G_k,\\
G_{k}(4)&=-G_{k}(3)+k3^{k-1}=k3^{k-1}-k2^{k-1}+k+G_k,\\
&\vdots\\
G_{k}(n)&=-G_{k}(n-1)+k(n-1)^{k-1}=k(n-1)^{k-1}-\cdots+(-1)^n G_k.
\end{aligned}
\]
Therefore,
\[
\sum_{i=2}^{n}G_k(i)=k(n-1)^{k-1}+k(n-3)^{k-1}+k(n-5)^{k-1}+\cdots\ .
\]
If $n=2m$,
\[
\sum_{i=2}^{2m}G_k(i)=k(2m-1)^{k-1}+k(2m-3)^{k-1}+k(2m-5)^{k-1}+\cdots+k+G_k.
\]
If $n=2m+1$,
\[
\sum_{i=2}^{2m+1}G_k(i)=k(2m)^{k-1}+k(2m-2)^{k-1}+k(2m-4)^{k-1}+\cdots+k2^{k-1}.
\]
Whether $n$ is odd or even, we always have the following trivial identity.
\[
\begin{split}
&G_k(1)+2\sum_{i=2}^{n}G_k(i)+G_{k}(n+1)\\
=&-G_k+k\Bigl[n^{k-1}+(n-1)^{k-1}+\cdots+1^{k-1}\Bigr]+G_k\\
=&k\sum_{i=1}^{n}i^{k-1}.
\end{split}
\]

By differentiating on $x$ at both sides of \eqref{eqn:Def_of_Dnx}, we also have
\[
G'_n(x)=nG_{n-1}(x).
\]
But being different from $\int_{0}^{1}B_n(x)dx=0$, we have
\[
\int_{0}^{1}G_n(x)dx=\int_{0}^{1}\frac{G'_{n+1}(x)}{n+1}dx=\frac{G_{n+1}(1)-G_{n+1}(0)}{n+1}=-\frac{2}{n+1}G_{n+1}.
\]

\begin{prop}
(i) $G_n(1-x)=(-1)^{n+1}G_n(x)$.\\
(ii) $G_n(x)=B_n(x)-2^n B_n(\frac{x}{2})$.\\
(iii) $G_n(x)=2^n B_n(\frac{x+1}{2})-B_n(x)$.\\
(iv) $B_n(x)=2^{n-1}\Bigl[B_n(\frac{x+1}{2})+B_n(\frac{x}{2})\Bigr]$.
\end{prop}
\begin{proof}
(i)
\[
\sum_{n=0}^{\infty}G_n(1-x)\frac{z^n}{n!}=\frac{ze^{(1-x)z}}{e^z+1}=-\frac{(-z)e^{x(-z)}}{e^{-z}+1}=\sum_{n=0}^{\infty}G_n(x)\frac{(-1)^{n+1}z^n}{n!},
\]
thus,
\[
G_n(1-x)=(-1)^{n+1}G_n(x).
\]

(ii)
\[
\frac{ze^{xz}}{e^z-1}-\frac{ze^{xz}}{e^z+1}=\frac{2ze^{\frac{x}{2}2z}}{e^{2z}-1}.
\]
By comparing the coeffients of $z^n$, we get
\[
B_n(x)-G_n(x)=2^n B_n(\frac{x}{2}).
\]

(iii)
\[
\frac{ze^{xz}}{e^z-1}+\frac{ze^{xz}}{e^z+1}=\frac{2ze^{\frac{x+1}{2}2z}}{e^{2z}-1}.
\]
By comparing the coeffients of $z^n$, we get
\[
B_n(x)+G_n(x)=2^n B_n(\frac{x+1}{2}).
\]

(iv) By (ii) and (iii).
\end{proof}
\begin{remark}
(1) Especially, we have $B_{2n}(\frac{1}{2})=4^n B_{2n}(\frac{1}{4})$, since $G_{2n}(\frac{1}{2})=0$.\\
(2) Let $x=0$ or $1$ in (iv), we have $B_{n}(\frac{1}{2})=(2^{1-n}-1)B_{n}$. Thus, $B_{2n}(\frac{1}{4})=2^{-2n}(2^{1-2n}-1)B_{2n}$.\\
(3) Let $x=1$ in (iii), we will get $G_n=(1-2^n)B_n$.
\end{remark}

Equation \eqref{eqn:G_n_and_B_n} can also be deduced in the following way. Using
\[
\frac{z/2}{e^{z/2}-1}=\frac{1}{2}\frac{ze^{z/2}+z}{e^z-1},
\]
we obtain
\begin{equation}\label{eqn:B_n_half}
B_n(\frac{1}{2})=(2^{1-n}-1)B_n,\quad\forall\ n\geqslant 0.
\end{equation}
Similarly,
\[
\begin{split}
\sum_{n=0}^{\infty}\frac{G_n}{n!}(\frac{z}{2})^n&=\frac{z/2}{e^{z/2}+1}=\frac{1}{2}\frac{ze^{z/2}-z}{e^z-1}\\
&=\frac{1}{2}\sum_{n=0}^{\infty}\frac{B_n(\frac{1}{2})}{n!}z^n-\frac{1}{2}\sum_{n=0}^{\infty}\frac{B_n}{n!}z^n.
\end{split}
\]
This infers that
\[
\frac{1}{2^n}G_n=\frac{1}{2}B_n(\frac{1}{2})-\frac{1}{2}B_n.
\]
By substituting \eqref{eqn:B_n_half} in the above formula, we obtainde
\[
G_n=(1-2^n)B_n,\quad n>1.
\]

%------------------------------
% input appendix and/or ack
%\input{other/appendix.tex}
%\input{other/ack.tex}

\noindent{\bf Acknowledgments :}
we express our gratitude to David Harvey who pointed out that the numbers $G_n$ in our manuscript (here is $G_n$) are essentially the Genocchi numbers, see \cite{GenocchiNumber}.

%-----------bibliography----------

%------------------------------------------
%Information of authors
\bigskip

\noindent Haifeng Xu\\
School of Mathematical Science\\
Yangzhou University\\
Jiangsu China 225002\\
hfxu@yzu.edu.cn\\
\medskip

\noindent Jiuru Zhou\\
School of Mathematical Science\\
Yangzhou University\\
Jiangsu China 225002\\
zhoujr1982@hotmail.com

\end{document}